\documentclass[12pt,reqno]{amsart}
\usepackage[centertags]{amsmath}
\usepackage{amsfonts,amssymb}
\usepackage{mathrsfs}
\usepackage{graphicx}

\theoremstyle{plain} 
\newtheorem{cor}{Corollary}

\newtheorem{thm}{Theorem}

\theoremstyle{definition}

\theoremstyle{remark}

\numberwithin{equation}{section}

 \DeclareMathOperator{\diag}{diag}

\DeclareMathOperator{\pd}{\partial}

\DeclareFontFamily{OT1}{pzc}{}
\DeclareFontShape{OT1}{pzc}{m}{it}{<-> s * [1.10] pzcmi7t}{}
\DeclareMathAlphabet{\mathpzc}{OT1}{pzc}{m}{it}

\DeclareMathSymbol{\R}{\mathalpha}{AMSb}{"52}
\DeclareMathSymbol{\C}{\mathalpha}{AMSb}{"43}
\newcommand{\mbb}[1]{\mathbb{#1}}

\newcommand{\K}{\mbb{K}}

\newcommand*{\Scale}[2][4]{\scalebox{#1}{$#2$}}%

\newcommand{\comment}[1]{}

\newcommand{\fq}{\mathfrak{q}}

\newcommand{\mathsc}[1]{{\normalfont\textsc{#1}}}
\newcommand{\scr}{\mathsc{r}}
\newcommand{\scs}{\mathsc{s}}

\newcommand{\by}{\mathbf{y}}
\newcommand{\bw}{\mathbf{w}}
\newcommand{\bh}{\mathbf{h}}

\newcommand{\mca}{\mathcal{A}}

\newcommand{\beqn} {\begin{equation} }
\newcommand{\eeqn}{  \end{equation}   }
\newcommand{\balign}{\begin{align}}
\newcommand{\ealign}{\end{align}}
\newcommand{\bsube}{\begin{subequations}}
\newcommand{\esube}{\end{subequations}}

\begin{document}

\title[]{Characterization of the class of canonical forms for systems of linear equations
}
\author[]{ JC Ndogmo}

\address{PO Box 728, Cresta 2118\\ South Africa
}
\email{ndogmoj@yahoo.com}

\begin{abstract}
The equivalence group is determined for systems  of linear ordinary differential equations in both the standard form and the normal form. It is then shown that the normal form of linear systems reducible by an invertible point transformation to the canonical equation $\by^{(n)}=0$ consists of copies of the same iterative equation. Other properties of iterative linear systems are also derived, as well as the superposition  formula for their general solution.
\end{abstract}

%
\maketitle

\section{Introduction}
\label{s:intro}
One of the major difficulties associated with the study of systems of linear ordinary differential equations, and in particular their symmetry properties is the large number of parameters that occur in their expression as arbitrary coefficients. Therefore, most of the general results concerning the symmetry properties and point transformations of this type of equations are limited to second- and third-order systems \cite{lopez-slde, fels-bk, waf-16melesh, bagderin-bk, campoa-14bk, bk014, melesh014}. Moreover, the study of these systems has also been often limited in the recent literature to the case of constant coefficients \cite{waf-13bk, campoa-16bk, bk014}.\par

In a slightly wider context, Gonz\'alez-Gasc\'on and Gonz\'alez-L\'opez estimated  in \cite{gonzal-newres} upper bounds $d_{\mathpzc{m},n}$ for the dimension of the symmetry algebra of systems of ordinary differential equations ({\sc ode}s) of the general form
\begin{equation} \label{gen-sde}
\mathbf{y}^{(n)} = F(x, \mathbf{y}_{(n-1)}), \qquad \mathbf{y} \in \R^{\mathpzc{m}},
\end{equation}
where $\mathbf{y}_{(n-1)}$ denotes $\mathbf{y}= \mathbf{y}(x)$ and all its derivatives up to the order $n-1$. It was thus found in \cite{gonzal-newres}  the exact bound value $d_{\mathpzc{m}, 2}= \mathpzc{m}^2 + 4\mathpzc{m} +3,$ and the more general value $d_{\mathpzc{m}, n}= \mathpzc{m}^2 + n \mathpzc{m} + 3$ for $n> 2.$ The value of $d_{\mathpzc{m},2}$ had been obtained earlier by Markus \cite{markus, bk014} in an independent study. All these results for $d_{\mathpzc{m},n}$ are in fact extensions of the well-known results obtained by Lie himself \cite{lie-dn}  for $d_{1, n},$ i.e. for scalar equations.\par

In the case of second-order or third-order equations of the form \eqref{gen-sde}, their group classification has been studied for linear equations with arbitrary coefficients \cite{waf-16melesh, melesh014}, and most often with constant coefficients \cite{waf-13bk, campoa-16bk, bk014}. In the case of arbitrary coefficients, this study has generally been reduced in the recent literature to systems of two equations while the dimension of the system is often arbitrary for equations with constant coefficients \cite{campoa-16bk, bk014}. Another important aspect of equations of the form \eqref{gen-sde} is their linearization criteria, and several independent results have also been recently obtained about these criteria for second-order equations \cite{fels-bk, aminova-bk, bagderin-bk, merker-bk}.

A study more directly related to the topic of this paper concerns systems of the form \eqref{gen-sde} with the right-hand side equal to zero, and which will henceforth be referred to as the canonical form. In the case of scalar equations, Lie \cite{liecanonic} showed that for $n \geq 2$ an $n$th order {\sc ode} is of maximal symmetry if it is reducible by an invertible point transformation to the canonical form . More recently, Krause and Michel \cite{KM} proved the converse of this result in the case of linear equations and showed that in fact a  linear scalar equation is of maximal symmetry if and only if it is iterative. Some recent results have also been obtained about  linear scalar equations of maximal symmetry, concerning their transformation properties \cite{JF, Jmaxv2}, their coefficient characterization \cite{ML, J14, JF}, their first integrals or Hamiltonian forms \cite{Jmaxv1, Jmaxv3}.\par

In this paper we completely derive the equivalence group for linear systems of {\sc ode}s in both the standard form and the  normal form, where normal form here means the standard form in which the coefficient of the term of order $n-1$ has vanished. We then use this result to show that a linear system of {\sc ode}s is reducible by an invertible point transformation to the canonical form if and only if its normal form consists of copies of iterative equations with the same source parameters. Several other results concerning systems of iterative equations and their general solutions  are also obtained. The base field for all differential equations considered will be assumed to be an arbitrary field $K$ of characteristic zero.
\section{Iterations of linear equations}
\label{s:iterat}
We shall denote by $F[x_1, \dots, x_p]$ a differential function of the  variables $x_1, \dots, x_p.$ Let $\Psi= \scr \frac{d}{dx} + \scs$ be an ordinary differential operator, where $\scr= \left(r_{ij}\right)$ and $\scs= \left( s_{ij} \right)$ are $\mathpzc{m} \times \mathpzc{m}$ matrices  with entries the smooth functions $r_{ij}$  and $s_{ij}$ of the independent variable $x \in \K,$ and where $\scr$ is invertible.  Linear iterative equations are the iterations $\Psi^n[\by, \scr, \scs]\equiv \Psi^n[\by] =0$ of the first-order linear ordinary differential equation $\Psi[\by] \equiv \scr \by'+ \scs \by=0$  where $\by = \by (x)\in \K^\mathpzc{m},$  and given by
\begin{equation} \label{iterdef}
\Psi^n [\by] = \Psi^{n-1}\left[\Psi \by \right],\quad \text{ for $n\geq 1,\quad$ where }\quad \Psi^0 =I
\end{equation}

is the identity operator. This definition is a natural extension to vector equations of that given for scalar equations for instance in \cite{KM, JF}. A linear iterative equation of a general order $n$ thus has the form
\begin{equation}
\label{geniter}
\Psi^n [\by] \equiv K_n^0\, \by^{(n)} + K_{n}^1\, \by^{(n-1)} +  K_{n}^2\, \by^{(n-2)} + \dots +  K_{n}^{n-1}\, \by' +  K_{n}^n\, \by =0,
\end{equation}

where the $K_n^j$ are $\mathpzc{m} \times \mathpzc{m}$ matrices.  By writing down the recurrence relation among the $K_n^j$ similar to that found for  iterative scalar  equations in \cite[P. 3]{JF}, the equality $K_n^0= \scr^n$ is immediately found. In the sequel $1\times 1$ matrices will be identified with scalars while for scalar equations we shall usually use the notations $r, s,$ and $ y$ instead of  $\scr, \scs,$ and $ \by$ corresponding to the vector equations  counterpart. In the case of scalar equations, since the coefficients are scalars and hence commutative,   the identities
\begin{subequations}\label{kn12}
\begin{align}
K_{n}^1 &= r^{n-1} \left[ n s + \binom{n}{2} r' \right] \label{kn12a} \\
K_n^2  &=  r^{n-2} \left[ \binom{n}{2} \Psi s + \binom{n}{3} \left( 3 s r' + r r''
    + \frac{3n-5}{4} r'^2 \right) \right] \label{kn12b}
\end{align}
\end{subequations}
are readily  established. Moreover, the well-known change of the dependent variable
\begin{equation} \label{std-2nor}
y = g w,\quad \text{where g satisfies }\quad K_{n}^1 \, g + n\, K_{n}^0\, g'=0
\end{equation}
maps \eqref{geniter} into its  normal form. This transformation however simply amounts to relating $r$ and $s$ in such a way that $K_n^1 =0.$ Therefore, for given source parameters $r$ and $s$ of the operator $\Psi,$ an $n$th order scalar linear equation in normal form
\begin{subequations} \label{iternor}
\begin{align}
&y^{(n)} + A_n^2\, y^{(n-2)} + \dots + A_n^j\, y^{(n-j)} + \dots A_n^n\, y=0  \label{iternor1}\\
\intertext{is iterative if and only if }\vspace{-10mm}
& A_n^j= \frac{K_n^j}{r^n}\Bigg \vert_{K_n^1=0} , \qquad (2 \leq j \leq n), \label{iternor2}
\end{align}
\end{subequations}
 It follows from \eqref{kn12a} that the requirement that $K_n^1=0$ holds is equivalent to having
\begin{equation} \label{s(x)}
s = - \frac{1}{2} (n-1) r',
\end{equation}
and this shows in particular that any iterative scalar equation in normal form can be expressed
in terms of the parameter $r$ alone, i.e. depends on a single arbitrary function. Moreover, for any function $\xi=\xi(x)$ which may also be a matrix function,  by setting
\begin{equation} \label{A(r)}
\mathcal{A}(\xi)(x) =   \frac{1}{4} [\xi(x)]^{-2} \left([\xi'(x)]^2 - 2 \xi(x) \xi''(x) \right)
\end{equation}
it follows from \eqref{kn12b} that in \eqref{iternor} we have
\begin{equation} \label{an2}
A_2^2 = \mathcal{A}(r), \quad \text{ and more generally }\quad A_n^2 = \binom{n+1}{3} \mathcal{A}(r).
\end{equation}
In fact, as already noted in \cite{ML, JF}, the coefficients $A_n^j$ depend only on the function $\mathcal{A}(r)= A_2^2$ and its derivatives. For simplicity, it will often be convenient to denote the coefficient $A_2^2$ of the term of third highest order in \eqref{iternor1} simply by $\fq.$  \par

If we set $r= u^2$ for a certain nonzero scalar function  $u,$ the expression for $\mathcal{A}(r)= \mathcal{A}(u^2)$ in \eqref{A(r)} is much simpler and reduces to $-u''/u.$ Setting $\fq = \mathcal{A}(r)$ is thus equivalent to letting $u$ be a nonzero solution of the equation
\begin{equation} \label{srce}
y'' + \fq y=0,
\end{equation}

which is referred to as the second-order source equation for \eqref{iternor}. Differential operators generating linear iterative equations in the normal form \eqref{iternor} with the coefficients $A_n^j=A_n^j[\fq]$ expressed directly as a function of $\fq$ and its derivatives have been obtained in \cite{JF} and an improved version appears in \cite{Jmaxv3}. The most general form of  linear iterative scalar equations in  normal form is given for example for the low orders three and four by
\begin{subequations}\label{o}
\begin{align}
0=& 2 \fq' y + 4 \fq y' + y''' = 0 \label{o3}\\
0=&10 y' \fq'+10 \fq y''+3 y \left(3 \fq^2+\fq''\right)+y^{(4)} \label{o4},
\end{align}
\end{subequations}
while the corresponding second-order equation with the same source parameter is given by \eqref{srce}.

\section{Equivalence group}
\label{s:transfo}
The group of equivalence transformations is well known for the general  linear $n$th order scalar {\sc ode} \cite{schwarz, ndogftc}. For homogeneous equations,  it consists of invertible point transformations of the form
\begin{subequations} \label{eqvgpK}
\begin{alignat}{3}
x &= f(z),&\qquad y &= g(z) w,&\qquad & \label{eqvgp-stdK} \\
\intertext{ for the standard form, and }
x &= f(z),&\qquad y&= c \left[f'(z)\right]^{(n-1)/2} w,&\qquad & \label{eqvgp-stdK}
\end{alignat}
\end{subequations}
 for the normal form, where $f$ and $g$ are arbitrary functions and $c$ an arbitrary nonzero scalar. For the general linear $n$th order vector {\sc ode} which without loss of generality may be assumed to be in the homogeneous form
\begin{equation} \label{glinS}
0 =\by^{(n)} + B_n^1\, \by^{(n-1)}+ \dots + B_n^k\, \by^{(n-k)} + \dots + B_n^n\, \by,
\end{equation}
where $\by \in \K^\mathpzc{m}$ and the $B_n^k=(B_{ij}^{k}(x))$ for $n$ fixed are $\mathpzc{m} \times \mathpzc{m}$ matrices, the corresponding extension of \eqref{eqvgp-stdK} can be guessed. But since such a result does not seem to be available in the literature, we completely derive it in the next theorem, as well as its normal form counterpart.


\begin{thm} \label{eqvthm}

The equivalence group of the general linear system \eqref{glinS} consists of invertible point transformations of the form
\begin{subequations} \label{eqvthm}
\begin{align}
x&= f(z), \qquad \by = T \bw.  \label{eqvthm1}
\intertext{ These transformations  reduce for the normal form of \eqref{glinS} in which the coefficient $B_n^1$ of the term of order $n-1$ has vanished to}
x&= f(z), \qquad \by = f'(z)^{(n-1)/2}\, C\, \bw.  \label{eqvthm2}
\end{align}
 Here, $T= \left(T_{ij} (z) \right)$ is an $\mathpzc{m} \times \mathpzc{m}$ matrix and $f$ as well as the entries $T_{ij}$ of $T$ are smooth functions, while $C=(C_{ij}) \in \K^{\mathpzc{m}^2}$ is a constant matrix.
\end{subequations}
\end{thm}
\begin{proof}
Let $x= f(z, \bw), \text{ and } \by = T (z, \bw),$  where $\bw=(w_1, \dots, w_\mathpzc{m})$, $f$ is a scalar function and $T= (T_1, \dots, T_\mathpzc{m})$ a vector function, be an invertible point transformation for \eqref{glinS}. The resulting transformation of \eqref{glinS} yields a polynomial in the $w_j$ and their derivatives, and the coefficients of the corresponding nonlinear monomials must vanish. If we denote by $p>1$ the highest exponent of $w_j'$ in this expansion, then the coefficient of $w_j'^{\,p}$ has a term of the form $T_1 B_{11}^n (f_{w_j})^p,$ where for a function $F=F(a_1, \dots, a_n)$ we use the notation $F_{a_i} = \frac{\pd F}{\pd a_i},$ $F_{a_i a_j} = \frac{\pd^2 F}{\pd a_i \pd a_j},$ etc. Due to the arbitrariness of the coefficient $B_{11}^n$ this term must vanish, and the invertibility of the transformation then implies that $f_{w_j}=0, \text{ and this holds for all }  j.$ Consequently, $f=f(z)$ depends on $z$ alone. Transforming again the equation with this new expression for $f,$ the vanishing of the coefficients of  monomials of degree two in the polynomial expansion of the transformed equation shows that
$$ (T_{i})_{w_k w_l} = 0,\qquad \text{for $i,k,l=0, \dots, \mathpzc{m}.$} $$
Consequently, we must have
$$
T_i (z, \bw) =  T_{ij}(z) w_j + \alpha_i(z), \qquad \text{ for $i=1, \dots, \mathpzc{m},$}
$$
for some functions $T_{ij}$ and $\alpha_i$ of $z.$ Moreover, from the fact that the equation is homogeneous it follows that $\alpha_i=0$ for all $i.$ This shows that an equivalence transformation of the equation must have the stated form \eqref{eqvthm1}. Conversely, a direct calculation shows that each invertible point transformation of the form \eqref{eqvthm1} leaves the family of equations \eqref{glinS} invariant, and thus \eqref{eqvthm1} represents indeed the a general element of the equivalence group of \eqref{glinS}. \par
When the general linear system \eqref{glinS} with $B_n^1=0$ is transformed by a group element of the form \eqref{eqvthm1}, the resulting coefficient $H_n^1$ of $\bw^{(n-1)}$ is an $\mathpzc{m} \times \mathpzc{m} $ matrix with entries
$$
H_{ij}^1= \frac{n\, T_{ij}'}{f'^{\,n}} - \frac{\binom{n}{2}\, T_{ij} f''}{f'^{\,n+1}},\quad \text{ for $i,j=1, \dots, \mathpzc{m}.$ }
$$
Solving each of the equations $H_{ij}^1=0$ for $T_{ij}$ shows that an equivalence transformation of the normal form counterpart of \eqref{glinS} must be of the form \eqref{eqvthm2}, and a direct calculation shows that each transformation of the stated form \eqref{eqvthm2} leaves indeed the normal form of \eqref{glinS} invariant, and this completes the proof of the Theorem.\par
\end{proof}

It should be noted that expressions for the scalar form \eqref{eqvgpK} and the vector form \eqref{eqvthm} of equivalence transformations are the same, and in particular \eqref{eqvgpK} is just as expected the one-dimensional version of \eqref{eqvthm}.

\section{The class of canonical forms}
In the study of symmetry properties of systems of linear second-order {\sc ode}s of maximal symmetry, Gonz\'alez-L\'opez showed in a seminal paper \cite{lopez-slde} that in its normal form $\by'' + A(x) \by=0,$ the coefficient $A(x)$ of $\by$ must be a scalar matrix. It was also proved in that paper that the second-order system of equations is of maximal symmetry if and only if its symmetry algebra is isomorphic to $\frak{sl} (\mathpzc{m} + 2, \R),$ assuming $\R$ to be the base field of the equation.\par

It was already known by then \cite{markus, gonzal-newres} that the maximal dimension of the symmetry algebra for second-order systems of the form \eqref{gen-sde} is $\mathpzc{m}^2+ 4\mathpzc{m} +3$ and attained for equations in canonical form $\by''=0.$ For linear systems \eqref{gen-sde}  of order greater than two, it seems to be known only for third-order equations \cite{fels-bk} that the maximal dimension is $\mathpzc{m}^2+ 3\mathpzc{m} +3$ and is reached only for equations equivalent by point transformations to the canonical form $\by'''=0.$\par

 The result of \cite{lopez-slde} about the normal form of second-order systems of maximal symmetry was obtained by a direct an intricate analysis of the determining equations of the symmetry algebra. By making use of equivalence groups, we shall give  a similar characterization for the class of equations \eqref{glinS}   reducible by an invertible point transformation to the canonical form. The class of canonical forms will henceforth be often referred to simply as the canonical class.\par

\begin{thm} \label{character}
 A linear $n$th order system of the form \eqref{glinS} is reducible by an invertible point transformation to the canonical form $\by^{(n)}=0$ if and only if its normal form has expression

\begin{subequations} \label{iterV}
\begin{align}
0 &=\bw^{(n)}+ A_n^2\,I_\mathpzc{m} \,  \bw^{(n-2)} + \dots + A_n^n\, I_\mathpzc{m}  \, \bw,  \label{itervec}\\
\intertext{where $I_\mathpzc{m}$ is the identity matrix of order $\mathpzc{m}$ and  the $A_n^j= A_n^j (z)$ are scalar functions such that}
0 &=w^{(n)}+ A_n^2\,w^{(n-2)} + \dots + A_n^n\, w,  \label{itersca}
\end{align}
\end{subequations}
is a scalar linear  iterative equation in the dependent variable $w=w(z).$
\end{thm}
\begin{proof}
Given that an equation of the form \eqref{glinS} is reducible by a point transformation to the canonical form if and only if its normal form can be obtained from the canonical form by a transformation of the form \eqref{eqvthm2}, to prove the theorem it suffices to show that the transformed version of the canonical form $\by^{(n)}=0$ under \eqref{eqvthm2} is of the required form  \eqref{itervec}. Under \eqref{eqvthm2} and multiplication from the left by the matrix $f'^{\frac{n+1}{2}} C^{-1},$  the canonical equation $y^{(n)}=0$  is transformed indeed into

\begin{subequations} \label{trsfcano}
\begin{align}
0 &=\bw^{(n)}+ A_n^2\,I_\mathpzc{m} \,  \bw^{(n-2)} + \dots + A_n^n\, I_\mathpzc{m}  \, \bw,  \label{trsfcanoV} \\
\intertext{where}
0&= w^{(n)} + A_n^2(z) w^{(n-2)} + \dots + A_n^n(z) w   \label{trsfcanos}
\end{align}
\end{subequations}
is exactly the corresponding transformation of the scalar version $y^{(n)}=0$ under \eqref{eqvgp-stdK}. But since by a result of \cite{KM}, a scalar linear equation is reducible by an invertible point transformation to the canonical form if and only if it is iterative, \eqref{trsfcanos} is iterative and this completes the proof of the theorem.
\end{proof}
Note that the factors $I_\mathpzc{m}$ occurring  in the right hand side of  \eqref{itervec} may clearly be omitted, and systems of equations of this form are termed \emph{isotropic}. Roughly speaking, Theorem \ref{character} states that the normal form of any linear equation in the canonical class consists of copies of the same iterative scalar  equation. This theorem shows both the strength of normal forms for systems of linear equations and the simplicity of the expressions of equations in the canonical class. Indeed, even in their normal form not all systems of linear equations \eqref{glinS} have the simple form \eqref{trsfcanoV} of a completely uncoupled system in which the coefficient matrices are not only diagonal, but scalar. Moreover, in its most general expression, the normal form of the linear system \eqref{glinS} depends on $(n-1) \times \mathpzc{m}^2$ arbitrary functions, while in \eqref{itervec} it depends only of $n-1$ arbitrary functions, and in fact just on $\fq$ and its derivatives. On the other hand, a system of linear equations reducible to canonical form by a point transformation need not be isotropic. Indeed, under the equivalence transformation \eqref{eqvthm1}, the transformed version of the generalized free fall equation $\by''=0$  takes the form
\begin{equation} \label{gtrans-o2}
0= \bw'' + \left[ 2 T^{-1} T' - \frac{f''}{f'} I_\mathpzc{m}  \right] \bw' + \left[ T^{-1} T'' - \frac{f''}{f'} T^{-1} T'  \right] \bw
\end{equation}

and thus, in standard form, it  is a coupled system that effectively depends on eight arbitrary functions, and not just formally on one such functions.\par

In terms of the iteration operator $\Psi$ in \eqref{iterdef}, Theorem \ref{character} says that an equation of the general form \eqref{glinS} belongs to the canonical class if and only if its normal form is an iterative equation in which the coefficient matrices $\scr= r I_\mathpzc{m},$ and $\scs = s I_\mathpzc{m}$ are scalar matrices and the scalar functions $r$ and $s$ are related as in \eqref{s(x)} by $s= -(n-1) r'/2.$
\begin{cor} \label{cornor}
A linear system of the form \eqref{glinS} belongs to the canonical class if and only if its normal form consists of uncoupled iterative equations of the form
\begin{equation} \label{eqcornor}
\Psi^n [y_i]= 0, \qquad \text{for $i=1, \dots, \mathpzc{m}$}
\end{equation}
where $\Psi= r(x) \frac{d}{dx} + s(x)$ is a differential operator and the scalar functions $r$ and $s$ satisfy $s= \frac{-1}{2}(n-1) r'.$
\end{cor}

\begin{proof}
This is just a restatement of the above remark according to which any equation in the canonical class must be iterative with the specified scalar matrices, and follows by expressing such a system componentwise.
\end{proof}

Although  an equation from the canonical class does not consist of uncoupled systems of identical iterative equations, it can be transformed into such a system.

\begin{cor} \label{corstd}
A linear system of the form \eqref{glinS}  belongs to the canonical class if and only if it can be transformed to a system of uncoupled iterative equations of the form
$$
\Psi^n [y_i]= 0, \qquad \text{for $i=1, \dots, \mathpzc{m}$}
$$
where $\Psi= r(x) \frac{d}{dx} + s(x)$ is a differential operator and the source parameters $r$ and $s$ are arbitrary scalar functions of $x.$
\end{cor}

\begin{proof}
Indeed, any equation of the general form \eqref{glinS} can be transformed into the normal form by a change of the dependent variable of the form $\by= Q(x) \bw$ where the invertible matrix $Q$ satisfies
\begin{equation}
 B_n^1 Q + n Q'=0 \label{tonor}.
\end{equation}
By Corollary  \ref{cornor}, the equation is in the canonical class if and only if  this transformed version consists of uncoupled iterative equations of the form \eqref{eqcornor}. Then, the additional  \emph{invertible} change of dependent variables $\bw= P(x) \bh,$ where $P = \exp\left[  \frac{1}{n}  \int_{x_0}^x  B(t) dt \right] I_\mathpzc{m}$ is a scalar matrix and $B(t)$ a scalar function will transform the uncouple system of iterative equations with same source parameters to another uncouple system of (identical) iterative equations with same source parameters in standard form, and this proves the corollary.
\end{proof}

It follows from Corollary \ref{corstd} that iterative equations generated by operators of the form $\Psi= \scr \frac{d}{dx} + \scs$ with source parameters the scalar matrices $\scr$ and $\scs$ belong to the canonical class. It is however clear that contrary to the scalar case, not all operators $\Psi$ generate equations belonging to the canonical class. Indeed, for a general operator $\Psi$ of the form \ref{iterdef} and $\by= \by(x) \in \K^\mathpzc{m}$ we have
\begin{subequations} \label{iterex}
\begin{align}
\Psi^2[\by] &= \scr \by'' + (\scr \scr' + \scr \scs + \scs \scr ) \by' + (\scr \scs' + \scs^2) \by,  \label{iterex2}\\
\begin{split}
\Psi^3[\by] &= \scr^3  \by''' \\
            &\quad +\big(\scr \scr' \scr + 2 \scr^2 \scr' + \scr^2  \scs + \scs \scr^2 + \scr \scs \scr\big) \by'' \\
            &\quad + \big(\scr \scr'^2 + \scr^2 \scr'' + \scr \scr' \scs + 2 \scr^2 \scs' + \scr \scs' \scr + \scr \scs \scr' \\
            &\quad + \scr \scs^2 + \scs^2 \scr + \scs \scr \scr' + \scs \scr \scs \big) \by' \\
            &\quad + \big(\scr \scr' \scs' + \scr^2 \scs'' + \scr \scs'\scs+ \scr \scs \scs' + \scs \scr \scs' + \scs^3\big) \by. \label{iterex3}
\end{split}
\end{align}
\end{subequations}
Using the fact that reducing the normalized iterative equation (with leading coefficient $I_\mathpzc{m}$) to normal form amounts to relating $\scr$ and $\scs$ in such a way that the coefficient of the term of second highest order vanishes, it turns out that condition \eqref{s(x)} also holds for commuting matrices, just because \eqref{kn12} holds for such matrices. That is, for commuting source parameters $\scr$ and $\scs,$ reducing the normalized equation  to normal form amounts to setting
\begin{equation} \label{s(x)mat}
\scs= -\frac{(n-1)}{2}  \scr'.
\end{equation}
In particular  $\scr$ and $\scr'$ should commute in such a case. We denote by
\begin{equation} \label{psinor}
\Psi_{\rm nor}= \scr \frac{d}{dx} - \frac{(n-1)}{2} \scr'
\end{equation}
the corresponding iteration operator. For $\mathpzc{m}= 2,$ the most general form of  matrices $\scr$ commuting with their derivative $\scr'$ is given by
\begin{equation} \label{aa'2}
\scr = \begin{pmatrix} a_1 & \lambda_1 a_2 \\  \lambda_2 a_2 & a_1 + \lambda_3 a_2   \end{pmatrix},
\end{equation}
where the $\lambda_j$ are arbitrary constants while the $a_i = a_i(x)$ for $i=1,2$ are some arbitrary functions. When condition \eqref{s(x)mat} is satisfied, the coefficient $\mathcal{A}(\scr)$ of $\by$ for the normal form of $\left( \Psi_{\rm nor}^2[\by]=0 \right)$ is clearly given by \eqref{A(r)}. However, for a matrix $\scr = \left( \begin{smallmatrix} \alpha & \beta \\  \gamma & \delta  \end{smallmatrix} \right)$ of the form \eqref{aa'2}, the entries of $\mathcal{A} (\scr) = (a_{ij})$ are given by
\begin{align} \label{A(r)com2}
\begin{split}
\begin{split}
 \hat{a}_{11} &= \Scale[0.90]{ \delta ^2 \left(\alpha'^2+\beta' \gamma'-2 \alpha  \alpha ''\right)-2 \beta ^2 \gamma  \gamma ''+\beta  \big(-\delta  \gamma
' \left(\alpha'+\delta'\right) }  \\
&\Scale[0.90]{ +\gamma  \left(\alpha'^2+\beta' \gamma'+2 \delta  \alpha ''\right)-\alpha  \left(\alpha' \gamma'+\gamma
' \delta'-2 \delta  \gamma ''\right)\big) }
\end{split} \\[1.5mm]
\begin{split}
\hat{a}_{12} &= \Scale[0.90]{ \delta ^2 \left(\alpha' \beta'+\beta' \delta'-2 \alpha  \beta ''\right)-2 \beta ^2 \gamma  \delta ''+\beta  \big(-\delta
 \left(\beta' \gamma'+\delta'^2\right)}\\
 &\Scale[0.90]{  +\gamma  \left(\alpha' \beta'+\beta' \delta'+2 \delta  \beta ''\right)-\alpha  \left(\beta
' \gamma'+\delta'^2-2 \delta  \delta ''\right)\big) }
\end{split}\\[1.5mm]
\begin{split}
\hat{a}_{21} &=\Scale[0.90]{ \gamma  \left(-\delta  \left(\alpha'^2+\beta' \gamma'\right)+\beta  \left(\alpha' \gamma'+\gamma' \delta'-2 \gamma  \alpha ''\right)\right)}\\
&\Scale[0.90]{ -\alpha
 \gamma  \left(\alpha'^2+\beta' \gamma'-2 \delta  \alpha ''-2 \beta  \gamma ''\right)+\alpha ^2 \left(\alpha' \gamma'+\gamma'
\delta'-2 \delta  \gamma ''\right)  }
\end{split} \\[1.5mm]
\begin{split}
\hat{a}_{22} &= \Scale[0.90]{ \gamma  \left(-\delta  \beta' \left(\alpha'+\delta'\right)+\beta  \left(\beta' \gamma'+\delta'^2-2 \gamma  \beta ''\right)\right)}\\
&\Scale[0.90]{ -\alpha
 \gamma  \left(\alpha' \beta'+\beta' \delta'-2 \delta  \beta ''-2 \beta  \delta ''\right)+\alpha ^2 \left(\beta' \gamma'+\delta
'^2-2 \delta  \delta ''\right),
}
\end{split}
\end{split}
\end{align}
where $\hat{a}_{ij} = 4 a_{ij}{(\beta  \gamma -\alpha  \delta )^2}.$  It is thus clear that even when the matrices $\scr$ and $\scs$ commute and satisfy \eqref{s(x)mat}, the corresponding matrix $\mathcal{A}(\scr)$ is not a scalar matrix as required for the normal form of the corresponding equation to lie in the canonical class. Moreover, the nonlinear system of equations $\hat{a}_{12}= 0=\hat{a}_{21}$  representing the necessary condition for $\mca(\scr)$ to be scalar is not obvious to solve directly. By considering some particular cases of this system and by denoting $u$ and $v$ arbitrary nonzero solutions of the second-order source equation  \eqref{srce}, it turns out that the resulting matrices $\scr$ satisfying \eqref{A(r)com2}, i.e. of the form \eqref{aa'2} and for which $\mca (\scr)$ is scalar are given by
\begin{subequations} \label{rok}
\begin{align}
\scr &= \begin{pmatrix}  u^2 & 0 \\ 0 & v^2  \end{pmatrix}  \label{rok1}
\intertext{when $\beta = \gamma=0,$ or }
\scr &= u^2  \begin{pmatrix}    k_1 & k_2 \\ k_3 & k_4  \end{pmatrix} \label{rok2}
\end{align}
\end{subequations}
when  the entries of $\scr$ are all scalar multiples of each other, and where the $k_j$ are arbitrary scalars. A direct calculation shows that equations  generated by iteration operators corresponding to these two matrices lie indeed in the canonical class.\par

On the other hand, for any choice of $\fq$ in \eqref{srce} and any operator $\Psi_{\rm nor}$ in \eqref{psinor} for which $\scr$ is a matrix of the form \eqref{rok}, when one applies to the generated equation the substitution
\begin{equation} \label{ruluj}
f^{(j)} =  D_x^{j-2} \left( - f \fq\right), \qquad \text{ for $j \geq 2,\;$ and $f= u, \text{ or } v,$}
\end{equation}
where $D_v= d/dv,$   one has exactly $\mathcal{A} (\scr) = \fq I_2.$ Consequently, given that the coefficient $\fq= A_2^2$ completely determines all scalar linear iterative equations in normal form and with second-order source equation \eqref{srce}, it follows from Corollary \ref{cornor} that for any order $n$ and under the substitution \eqref{ruluj}, one must have
\begin{subequations} \label{iterok}
\begin{align}
\scr^{-n}\, \Psi_{\rm nor}^n[\by] &= \by^{(n)} + A_n^2  \by^{(n-2)} + \dots + A_n^n  \by=0,  \label{iterokV}\\
\intertext{where the scalar equation }
y^{(n)} &+ A_n^2  y^{(n-2)} + \dots + A_n^n y =0,  \label{iterokS}
\end{align}
\end{subequations}
is the corresponding $n$th order iterative equation with source equation \eqref{srce}, and more precisely  one has $A_n^j= A_n^j[\fq].$ That is, the coefficients $A_n^j$ in  the iterative vector equations of the form \eqref{iterokV} are obtained explicitly as differential functions of $\fq.$ In the sequel  $\diag (x_1, \dots, x_p)$ will denote a diagonal matrix with diagonal entries $x_1, \dots, x_p.$\par

 Although the matrices $\scr$ in \eqref{rok} have been calculated explicitly only for the dimension $\mathpzc{m}=2,$ it is however easy to verify  directly that for  any square matrix of order $\mathpzc{m} \geq 1$ of the form
$$
\scr =\diag \left(u_1^2, \dots, u_\mathpzc{m}^2 \right),\quad \text{ or }\qquad   \scr = u^2 \left(  r_{ij}\right),
$$
where $u$ and the $u_j$ for $j=1,\dots \mathpzc{m}$ are arbitrary nonzero solutions of the second-order source equation \eqref{srce} while the $r_{ij}$ are scalars, the matrix $\scr$ commutes with its derivative and $\mca(\scr)= \fq I_\mathpzc{m}.$ Consequently, for all orders $n \geq 1$ any iterative equation generated by the corresponding operator $\Psi_{\rm nor}$ of the form \eqref{psinor} lies in the canonical class.\par

Due to the fact that the normal form of equations in the canonical class consists of uncoupled iterative equations, their general solutions can also be obtained from those of the second-order source equation \eqref{srce} by a simple superposition principle. Indeed, the general solution $\bw= (w_1, \dots, w_\mathpzc{m})$ of a system of linear equations of the form \eqref{itervec} is given simply in terms of two linearly independent solutions $u$ and $v$ of \eqref{srce} by
\begin{equation} \label{gensol1}
w_i = \sum_{j=1}^n C_{ij} u^{n-j} v^{j-1}, \qquad \text{for $i=1,\dots,\mathpzc{m},$}
\end{equation}
where the $C_{ij}$ are arbitrary scalars. Indeed, this follows simply from the fact that \eqref{itervec} consists of identical scalar equations of the form \eqref{itersca}, the general solution of which is given by \eqref{gensol1} for a fixed value of $i.$  In fact, as demonstrated in \cite{Jmaxv3} it suffices to know a single nonzero solution $u$ of the source equation \eqref{srce} to derive the general solution of \eqref{itervec}, as the general solution $\bw$ can indeed also be expressed as
\begin{equation} \label{gensol2}
w_i = \sum_{j=1}^n C_{ij} u^{n-1} \left( \int \frac{d x}{u^2} \right)^j, \qquad \text{for $i=1,\dots,\mathpzc{m}.$}
\end{equation}
Similarly, for a general equation of the form \eqref{glinS} in the canonical class, once we know the matrix $Q$ in \eqref{tonor} defining the transformation mapping it to its normal form, the general solution of \eqref{glinS} is simply given by $\by = Q \bw,$ with $\bw$ given by \eqref{gensol1}.

\section{Concluding remarks}
We have shown in this paper that any system of linear equations reducible by an invertible point transformation to the canonical form is iterative, and determined arbitrary classes of iteration operators generating equations in the canonical class. However, we've also noted that contrary to the scalar case proved in \cite{KM}, not all iterative systems of linear equations belong to the canonical class. It is therefore an open problem to find the most general form of operators $\Psi= \scr \frac{d}{dx} + \scs$ whose generated equations all lie in the canonical class.\par

If it was known as it is in the case of scalar equations that a system of linear equations is reducible by an invertible point transformation to the canonical equation if and only if it is of maximal symmetry, then the result of this paper asserting that any linear equation in the canonical class is iterative would be an extension of that obtained by Krause and Michel \cite{KM} for scalar equations. However, to our knowledge it has been established only for linear systems of orders two \cite{lopez-slde} and three \cite{fels-bk} that the canonical class and the class of linear systems of maximal symmetry are the same.

\end{document}